\documentclass[a4paper,12pt]{article}
\usepackage{inputenc}
\usepackage{amsmath}
\usepackage{graphics}
\usepackage{epsfig}
\usepackage{graphicx}
\usepackage{latexsym}
\usepackage{graphics}
\usepackage{epsfig}
\usepackage{amsmath,amssymb,amsfonts,theorem,color,bm}
\usepackage{multirow}
\usepackage{verbdef}
\usepackage{mathrsfs}
\newcommand{\beqn}{\begin{equation}}
\newcommand{\eeqn}{\end{equation}}

\newtheorem{definition}{Definition}[section]
\newtheorem{remark}{Remark}[section]

\newtheorem{theorem}{Theorem}[section]

\newtheorem{proposition}[theorem]{Proposition}

\def\blackbox{\leavevmode\vrule height 5pt width 4pt depth 0pt\relax}
\newenvironment{proof}{\begin{trivlist}
		\item[]\hspace{0cm}{\bf Proof:}\hskip -5pt
		\hspace{0cm} }{\hfill $\blackbox$
	\end{trivlist}}

		\title{Rational RBF-based partition of unity method for efficiently and accurately approximating  3D objects}

\begin{document}
		\begin{center}
			\section*{Rational RBF-based partition of unity method for efficiently and  accurately approximating  3D objects}
		\end{center}
		\vskip 0.3cm
		\begin{center}
			{\large {E. Perracchione}}
			\vskip 0.3cm
			Dipartimento di Matematica \lq\lq Tullio  Levi-Civita\rq\rq, Universit\`a di Padova, Italy\\ 
			\vskip 0.2cm
			{\tt emma.perracchione@math.unipd.it}
			
		\end{center}
			
			\section*{Abstract}

We consider the problem of reconstructing 3D objects via meshfree interpolation methods. In this framework, we usually deal with large data sets and thus we develop an efficient local scheme via the well-known Partition of Unity (PU) method. The main contribution in this paper consists in constructing the local interpolants for the implicit interpolation by means of Rational Radial Basis Functions (RRBFs). Numerical evidence confirms that the proposed method is particularly performing when 3D objects, or more in general implicit functions defined by scattered data, need to be approximated. 

\section{Introduction}

The problem of reconstructing 3D objects is   common  in computer aided  design  and computer graphics. Truly performing mesh-dependent approaches have already been developed in this context (for a general overview on these approaches refer e.g. to \cite{Peigl}). This computational issue leads to the approximation of surfaces defined in terms of \emph{point cloud data}, i.e. a set of unorganized points in 3D. Since such data sets are usually {large}, it might be convenient to investigate, as in this paper,  meshfree approaches. Specifically, here we focus on Radial Basis Function (RBF) interpolants; refer e.g. to \cite{Fasshauer15,Fasshauer,Wendland05}. Because of the huge amount of data that are usually involved in the reconstruction of 3D objects, we decompose the original problem, which leads to solving  \emph{large} linear systems, into many small ones. This can be efficiently done by means of the Partition of Unity (PU) method (see e.g. \cite{Wendland02a}). In this way, the original reconstruction domain is split into many \emph{subdomains} or \emph{patches} and, as a consequence, only linear systems of moderate size need to be solved. 

However, the local approximants might suffer from instability and/or the local function values might oscillate defining steep gradients. These are the main reasons for which we consider, for each subdomain,  a Rational RBF (RRBF) expansion. The RRBF have been introduced in \cite{Jakobsson} (see also  \cite{rat0}) and further developed for collocation methods in \cite{sarra1}. The proposed scheme reduces to a largest eigenvalue problem which is efficiently solved by means of the so-called Deflation Accelerated Conjugate Gradient (DACG) algorithm (see e.g. \cite{bgp97nlaa}). Furthermore, we also provide pointwise error bounds for the local RRBF interpolants.

This investigation reveals that, for the reconstruction of 3D objects, the new method, namely RRBF-PU, performs better than the classical approach, i.e. the one based on local RBF interpolants \cite{Fasshauer}.  Furthermore,  comparisons with the BLOOCV-PU scheme, which is based on selecting \emph{optimal} local RBF interpolants, will be also carried out \cite{cavoretto_loocv}. 

To be more precise, for the approximation of point cloud data, we first need to define additional interpolation conditions and function values, so that we reduce to a standard interpolation issue on $\mathbb{R}^3$. This is also known as the \emph{implicit}  approach and consists in adding extra interpolation conditions; see e.g. \cite{carr01,Fasshauer}. 
Let us fix the space dimension $M=3$. Given a point cloud data set  $ {\cal X}_n= \{ \boldsymbol{x}_i \in \mathbb{R}^{M},$ $ i=1, \ldots, n \}$, that describes a surface on $\mathbb{R}^M$, i.e. a two dimensional manifold $\mathscr{S}$, the aim consists  in finding an approximate surface $ \mathscr{S}^{*}$. In this framework,  $\mathscr{S}$ is defined by all  points  $\boldsymbol{x}  \in \mathbb{R}^{M}$ such that
$
f\left(\boldsymbol{x} \right)=0,
$
for some (unknown) function $f$. To approximate the implicit function $f$, the trick is the one of adding an extra set of points, namely \emph{off-surface points}, so that we can then compute a three dimensional interpolant via the PU method by using local RRBFs.

Thus, the reconstruction scheme considered in this paper can be summarized in the following three steps (see e.g. \cite{Perracchione}):
\begin{enumerate}
	\item[i.] generate the extra off-surface points (see Section \ref{sec1}),
	\item [ii.] construct the PU structure for the local approach (see Section \ref{sec:2}),
	\item [iii.] compute a local RRBF approximant for each subdomain (see Section \ref{rat}).
\end{enumerate}

This scheme is tested via extensive numerical experiments carried out in  Section \ref{ne}. 
Finally, Section \ref{concl} is devoted to conclusions and future work.

\section{Extra off-surface points}
\label{sec1}	
To construct the additional data \cite{hoppe94,hoppe92}, we need to suppose that for each point $ \boldsymbol{x}_i \in {\cal X}_n$, its oriented normal $\boldsymbol{n}_i \in \mathbb{R}^{M} $ is known. If normals are not available, we describe a technique to approximate them in Subsection \ref{suseq1}.

\subsection{The implicit approach}

In practice, as usually done in literature (see e.g. \cite{Cuomo,Fasshauer,hoppe94}), for each  $ \boldsymbol{x}_i \in {\cal X}_n$, we compute the following two additional points:
\begin{equation*}
\boldsymbol{x}_{n+i}=\boldsymbol{x}_i+ \Delta_i \boldsymbol{n}_i,
\quad \textrm{and} \quad 
\boldsymbol{x}_{2n+i}= \boldsymbol{x}_i- \Delta_i \boldsymbol{n}_i,
\end{equation*}
where $\Delta_i > 0\in \mathbb{R}^{+}$, $i=1,\ldots,n$, are  chosen  stepsizes along the normals $\boldsymbol{n}_i$. In other words, we provide two other sets of points, namely ${\cal X}_{\Delta}^{+}=$ $  \{ \boldsymbol{x}_{n+1},\ldots, \boldsymbol{x}_{2n} \}$ and ${\cal X}_{ \Delta}^{-}=$ $ \{ \boldsymbol{x}_{2n+1},\ldots,$ $\boldsymbol{x}_{3n} \}$. Then, we define ${\cal X}_N= {\cal X}_{n} \cup {\cal X}_{\Delta}^{+} \cup {\cal X}_{\Delta}^{-}$ as the set of all  points on which the interpolation conditions are given. 

\begin{remark}
In what follows, for a given $\Delta > 0\in \mathbb{R}^{+}$, we simply fix $\Delta_i = \Delta$, $i=1,\ldots,n$. More precisely, to possibly avoid the effect of ill-conditioning, one can construct the extended data set with a stepsize $\Delta$ proportional to the so-called separation distance $q_{ {\cal X}_n}$ of the original data set $ {\cal X}_n$, where:
\begin{equation}
	q_{ {\cal X}_n} = \dfrac{1}{2} \min_{  i \neq k} \left\| \boldsymbol{x}_i - \boldsymbol{x}_k \right\|_2.
	\label{sd}
\end{equation}
 In particular, we take $\Delta= \xi q_{ {\cal X}_n}$, with $\xi=1/3$. This choice follows from the investigation carried out in \cite{Cuomo1}. Indeed, the authors prove that, for a suitable choice of the stepsizes  $\Delta_i > 0\in \mathbb{R}^{+}$, $i=1,\ldots,n$ and for $0<\xi\leq 1/3$, we have $q_{ {\cal X}_N}= \xi q_{ {\cal X}_n}$. Therefore, since the ill-conditioning grows with the decrease of  $q_{ {\cal X}_N}$, taking {small} values for $\xi$ is not recommended. On the other hand, as shown in \cite{Cuomo1}, for $\xi>1/3$ the added points might be close to each other and this might lead to the self-intersecting surfaces phenomenon. Therefore, fixing $\xi=1/3$ seems to provide a good compromise among accuracy and stability.
\end{remark}

Moreover, we construct the augmented set of function values ${\cal F}_{N}$. It is defined as the union of the following sets \cite{Fasshauer}
\begin{equation*}
\begin{array}{lll}
\vspace{0.1cm}
{\cal F}_{n} = \{f_i \hskip 0.1cm :\hskip 0.1cm   f(\boldsymbol{x}_i)=a, \hskip 0.1cm i =1, \ldots ,n\},\\
\vspace{0.1cm}
{\cal F}^{+}_{\Delta}  = \{f_i \hskip 0.1cm :\hskip 0.1cm f(\boldsymbol{x}_i)=b, \hskip 0.1cm i =n+1, \ldots ,2n\},\\
{\cal F}^{-}_{\Delta}  = \{f_i \hskip 0.1cm :\hskip 0.1cm f(\boldsymbol{x}_i)=c, \hskip 0.1cm i = 2n+1, \ldots ,3n \},
\end{array}
\end{equation*}
where the values of $a$, $b$ and $c \in \mathbb{R}$ are arbitrarily and usually set as $0$, $1$ and $-1$, respectively.
Now, after creating the data set, we  are able to compute an interpolant   whose $a$-contour (iso-surface) interpolates the given point cloud data.

However, note that to construct the augmented data set, we assume to know the normals to the implicit surface at each point. Unfortunately, in applications such normals are usually unknown and thus need to be estimated. For this reason, we illustrate a technique to calculate them.  

\subsection{Normals estimation}
\label{suseq1}

Let us fix a number $K<n$ and compute for each point $\boldsymbol{x}_i$, $i=1,\ldots,n$, its $K$  nearest neighbors set $\mathscr{K}_i$. Then, following the technique presented in  \cite{hoppe94,hoppe92}, for each data point  $\boldsymbol{x}_i$, $i=1,\ldots,n$, we compute a local oriented tangent plane $\mathscr{T}_i$. The latter is defined by a point, called  centre  $\boldsymbol{c}_i$,  and a unit normal
vector $\boldsymbol{n}_i$. More precisely, we have that
\begin{equation*}
\boldsymbol{c}_i= \frac{1}{K} \sum_{k \in \mathscr{K} ( \boldsymbol{x}_i)} \boldsymbol{x}_k.
\end{equation*}

Furthermore, since the normal  $\boldsymbol{n}_i$ is computed via  Principal Component Analysis (PCA), see e.g. \cite{Belton}, we evaluate the following covariance matrix $V \in \mathbb{R}^{3 \times 3}$ defined by
\begin{equation*}
V( \boldsymbol{x}_i)= \sum_{k \in \mathscr{K} ( \boldsymbol{x}_i)} (\boldsymbol{x}_k - \boldsymbol{c}_i)(\boldsymbol{x}_k - \boldsymbol{c}_i)^{T}.
\end{equation*}
It is trivially symmetric and positive semi-definite. Its eigenvalues $\lambda_{i1} \geq \lambda_{i2} \geq \lambda_{i3}$ and  corresponding  unit eigenvectors $\boldsymbol{v}_{i1} ,\boldsymbol{v}_{i2},\boldsymbol{v}_{i3}$ represent the plane  and the normal to such plane. If two eigenvalues, for instance $\lambda_{i1}$ and $\lambda_{i2}$, are close together and the third one is significantly smaller, then the eigenvectors for the first two eigenvalues 	$\boldsymbol{v}_{i1}$ and $\boldsymbol{v}_{i2}$  determine the plane, while  $\boldsymbol{v}_{i3}$  is the corresponding normal.

Unfortunately, this is not sufficient to construct the data set. In fact,  we need to orient the normals. 

We consider here the solution already proposed in \cite{hoppe94,hoppe92}, which consists in building, at first, the \emph{Riemann graph}  $G= \{ \cal{V}, \cal{E} \} $, with each node in $ \cal{V} $ corresponding to one of the 3D data points. To be more precise, in our case it is an undirected graph that has a vertex for every normal  $ \boldsymbol n_i$  and an edge $e_{ik}$  between the vertices of  $ \boldsymbol n_i$ and  $ \boldsymbol n_k$ if and only if $ i \in \mathscr{K}( \boldsymbol{x}_k)$ or  $ k \in \mathscr{K}( \boldsymbol{x}_i)$.

Therefore, to orient the normals, the idea consists in starting with an arbitrary normal orientation and then to propagate such orientation among neighboring points.  We assign to each edge $e_{ik}$ the cost 
\begin{equation*}
w(e_{ik})=  1-| \boldsymbol{n}_i \boldsymbol{n}^{T}_k|.
\end{equation*}
Since  $w(e_{ik})$  is small if the unoriented tangent planes are nearly parallel, we can propagate the orientation  by traversing the \emph{minimal spanning tree} of the Riemann graph.

For completeness, we recall  some  definitions (see e.g. \cite{Behzad} for further details).

\begin{definition}
	In any connected graph $ G$, a spanning tree is a subgraph of $ G$ having the following two properties:
	\begin{enumerate}
		\item[i.]  the subgraph is a tree, 
		\item[ii.]  the subgraph contains every vertex of $G$.
	\end{enumerate}
\end{definition}
\begin{definition}
	The weight of a tree is the sum of the weights of all edges in the tree.
\end{definition}
\begin{definition}
	Given a connected weighted graph $ G$ the  minimal spanning tree is the one having minimum weight among all spanning trees in the graph.
\end{definition}
Thus, to propagate the normal orientation,  we begin by choosing  an edge of minimum weight in the graph and  we then continue by selecting from the remaining edges an edge of minimum weight  until a spanning tree is formed. This scheme is known as Kruskal's algorithm and the reader can, for instance, refer to \cite{Gould} for further details.

To summarize, starting from the initial point cloud data set, we now obtain the set of nodes $ {\cal X}_N= \{ \boldsymbol{x}_i \in \mathbb{R}^{M},$ $ i=1, \ldots, N \}$ and the one of function values $ {\cal F}_N= \{f_i \in \mathbb{R},$ $ i=1, \ldots, N \}$. In other words, we reduce to  a \emph{standard} 3D interpolation problem which will be solved by means of the RRBF-PU method, as described in the next section.
Note that, we end up with a data set ${\cal X}_N$ that contains \emph{about} three times the number of points of the original one ${\cal X}_n$. In general we might have $N \neq 3n$. Indeed, we must exclude points that have zero normals.

\section{The partition of unity structure}	
\label{sec:2}

The PU method takes advantage of being a local technique, so that we always need to deal with {small} linear systems. It finds its origin around 1960 (see \cite{Shepard68a}) and is also well-known in the context of Partial Differential Equations (PDEs) (see e.g. \cite{Babuska-Melenk97,Larsson-Lehto}).

\subsection{Remarks on radial basis function interpolants}

In order to introduce the RRBF-PU  interpolation, we need to remark the main features of the standard RBF approximation theory. Thus, let 
$\Omega \subseteq \mathbb{R}^M$ be a bounded set,  $ {\cal X}_N = \{  \boldsymbol{x}_i, i = 1,  \ldots , N \} \subseteq \Omega$ the set of nodes  and $ {\cal F}_N= \{ f_i = f(\boldsymbol{x}_i) ,i=1, \ldots, N \}$ the set of function values, as defined in the previous section. A global interpolant  ${\cal R}: \Omega \longrightarrow \mathbb{R}$ is such that	
\begin{equation}
{\cal R}\left( \boldsymbol{x}_i\right)=f_i, \quad i=1, \ldots, N.
\label{int12}
\end{equation}	
Here we take ${\cal R} \in H_{\Phi} ({\cal X}_N)= \textrm{span} \{ \Phi(\cdot,\boldsymbol{x}_i), \boldsymbol{x}_i \in {\cal X}_N\}$, where $\Phi : \Omega \times \Omega \longrightarrow \mathbb{R}$ is a strictly positive definite and symmetric kernel \cite{Fasshauer15,Wendland05}. With this choice the interpolant \eqref{int12} assumes the form
\begin{equation*}
{\cal R}(\boldsymbol{x}) = \sum_{k = 1}^N \alpha_k \Phi(\boldsymbol{x}, \boldsymbol{x}_k), \quad \boldsymbol{x}\in\Omega.
\end{equation*}
Therefore, to determine the coefficients $\boldsymbol{\alpha}= (\alpha_1, \ldots, \alpha_N)^T$, one needs to solve $A \boldsymbol{\alpha}= \boldsymbol{f}$, where the entries of the matrix $A \in \mathbb{R}^{N \times N}$ are given by 
\begin{equation}
(A)_{ik}= \Phi (\boldsymbol{x}_i , \boldsymbol{x}_k), \quad i,k=1, \ldots, N,
\end{equation}
and  $\boldsymbol{f}= (f_1, \ldots, f_N)^T$. Existence and uniqueness of the solution are ensured by the fact that the kernel $\Phi$ is strictly positive definite and symmetric \cite{Fasshauer15,Wendland05}. Since here we take RBFs, we also have to take into account the \emph{shape parameter}, i.e. we assume that there exist  a function $ \phi: [0, \infty) \to \mathbb{R}$ and a shape parameter $\varepsilon>0$ such that 
\begin{equation*}
\Phi(\boldsymbol{x},\boldsymbol{y})=\phi_{\varepsilon}( ||\boldsymbol{x}-\boldsymbol{y}||_2):=\phi(r),
\end{equation*}
for all $\boldsymbol{x},\boldsymbol{y} \in \Omega$. To be more precise, since $\varepsilon$ is a scalar, we should refer to $\Phi$ as isotropic kernel. However, to simplify the notation, we will omit the term isotropic.

Note that, for each positive definite and symmetric kernel $\Phi$, we are able to associate  the so-called \emph{native space} ${\cal N}_{\Phi} (\Omega)$. To point out this fact,  we first introduce the following pre-Hilbert space with reproducing kernel $\Phi$  \cite{Fasshauer}
\begin{equation*}
H_{\Phi}(\Omega)=\textrm{span} \{ \Phi\left(\cdot,\boldsymbol{x}\right), \boldsymbol{x} \in \Omega\},
\end{equation*}
with the associated bilinear form $\left(\cdot,\cdot\right)_{H_{\Phi}(\Omega)}$ given by
\begin{equation*}
\left( \sum_{i=1}^l \alpha_i \Phi\left(\cdot,\boldsymbol{x}_i\right),  \sum_{k=1}^l \beta_k \Phi\left(\cdot,\boldsymbol{x}_k\right) \right)_{H_{\Phi}(\Omega)}= \sum_{i=1}^l \sum_{k=1}^l \alpha_i \beta_k \Phi\left(\boldsymbol{x}_i,\boldsymbol{x}_k\right),
\end{equation*}
where $l= \infty$ is also allowed.
Since $H_{\Phi}(\Omega)$ is only pre-Hilbert we define the native space ${\cal N}_{\Phi}(\Omega)$ of $\Phi$ to be the completion of $H_{\Phi}(\Omega)$ with respect to the norm $||\cdot||_{H_{\Phi}(\Omega)}$ so that $||f||_{H_{\Phi}(\Omega)} = ||f||_{{\cal N}_{\Phi}(\Omega)}$, for all $f \in H_{\Phi}(\Omega)$, see \cite{Fasshauer,Wendland05}.

Finally, to consistently introduce RRBFs, we have to define the so-called \emph{fill distance} and report the following theorem on polynomial precision (cf. \cite[Th. 3.14, p. 33]{Wendland05}).

The fill distance is defined as
\begin{equation*}
h_{{\cal X}_N} :=  \sup_{ \boldsymbol{x} \in \Omega} \left( \min_{ \boldsymbol{x}_k  \in {\cal X}_N} \left\| \boldsymbol{x} - \boldsymbol{x}_k \right\|_2 \right),
\label{fd}
\end{equation*}
and represents  the radius of the largest possible empty ball that can be placed among the data locations inside $\Omega$. We also remark that, as the fill-distance diminishes the interpolation error decreases, provided that sufficiently stable methods or RBFs with limited regularity are used (see e.g. \cite{Demarchi15,Fornberg}). Otherwise it is well-known that we might have numerical instability due to ill-conditioning. 

\begin{theorem} 
	Suppose that $ \Omega \subseteq  \mathbb{R}^M$ is compact and satisfies an	interior cone condition with angle $ \theta = \left(0, \pi /2 \right)$ 	and radius  $\gamma>0$. Fix $l \in \mathbb{N}$ and let $ \Pi_{l-1}^{M}$ be the set of polynomials of degree $l-1$.
	Then,  there exist $ h_0$, $C_1$, $C_2 > 0$	constants depending only on $l$, $ \theta$ and $\gamma$, such that for every $ {\cal X}_N= \{ \boldsymbol{x}_i  ,i=1, \ldots, N \} \subseteq \Omega$ with 	$h_{  {\cal X}_N} \leq h_0$ 
	and every $\boldsymbol{x} \in \Omega$, we can find real  numbers 	$v_k(\boldsymbol{x})$, $ k=1, \ldots, N, $ such that:
	\begin{itemize}
		\item[i.] $ \sum_{k=1}^{N} v_k \left(\boldsymbol{x} \right) p\left( \boldsymbol{x}_k\right)= p\left( \boldsymbol{x}\right)$, for all $p \in  \Pi_{l-1}^{M} $,
		\item[ii.] $  \sum_{k=1}^{N} | v_k \left(\boldsymbol{x} \right) | \leq C_1 $,
		\item[iii.] $v_k\left( \boldsymbol{x}\right)=0 $ provided that $\left\| \boldsymbol{x}- \boldsymbol{x}_k\right\|_2 \geq C_2 h_{  {\cal X}_N}$.
	\end{itemize}
	\label{tha}
\end{theorem}

The method described in this subsection is effective only if we deal with data sets of moderately large sizes, otherwise the computational cost of computing the inverse of large interpolation matrices is prohibitive. Thus, in the next subsection, we introduce the PU method, that suitably works for huge sets of points.		

\subsection{The PU method}

For the PU scheme, we first need to divide the domain $\Omega$ into $d$ subdomains  $ \Omega_j$,  such that $ \Omega  \subseteq \cup_{j=1}^{d} \Omega_j$. As a consequence, they have to satisfy some mild overlap condition. Further, we require that such covering is also regular \cite{Wendland02a}.

\begin{definition}
	Suppose that $ \Omega \subseteq  \mathbb{R}^M$ is bounded and $ {\cal X}_N= \{ \boldsymbol{x}_i  ,i=1, \ldots, N \} \subseteq \Omega$ is given. An open and bounded covering $ \{ \Omega_j \}_{j=1}^{d}$ is called regular for $( \Omega, {\cal X}_N)$ if the following properties are satisfied:
	\begin{itemize}
		\item[i.] for each $ \boldsymbol{ x} \in \Omega$, the number of subdomains $ \Omega_j$, with $ \boldsymbol{x} \in \Omega_j$, is  bounded by a global constant $d_0$,
		\item[ii.] the local fill distances $ h_{ {\cal X}_{N_j}}$ are uniformly bounded by the global fill distance $h_{{\cal X}_N}$, where ${\cal X}_{N_j}= {\cal X}_N \cap \Omega_j$.
		\item[iii.] there exists $C_r>C_2$ such that each subdomain $ \Omega_j$ satisfies an interior cone condition with angle $ \tilde{\theta} \in (0, \pi/2) $ and radius $\tilde{\gamma}= C_r  h_{ {\cal X}_{N}}$. 
	\end{itemize}
\end{definition}

Once we have such covering of the domain, we construct for each subdomain a local interpolant. Then, the local fits are glued together by means of $d$ weight functions $W_j$, $j=1,\ldots,d$, i.e. the PU interpolant ${\cal I}$ assumes the form:
\begin{equation*}
{\cal I}\left( \boldsymbol{x}\right)= \sum_{j=1}^{d} L_j\left( \boldsymbol{x} \right) W_j \left( \boldsymbol{x}\right), \quad \boldsymbol{x} \in \Omega,
\end{equation*}
where $L_j$, $j=1,\ldots, d$, are local interpolants such as  RBF approximants defined in the previous subsection.
Before discussing which local interpolants $L_j$, $j=1,\ldots,d$, are used,  we first recall that $ \{ W_j \}_{j=1}^{d}$ must form a $k$-stable partition of unity, i.e. they form a family of compactly supported, non-negative, continuous functions such that
\begin{itemize}
	\item[i.] $\text{supp}\left(W_j\right) \subseteq \Omega_j$,	
	\item[ii.] 	$\sum_{j=1}^{d} W_j\left(\boldsymbol{x}\right) = 1, \quad \boldsymbol{x} \in \Omega$,
	\item[iii.] for every $ \boldsymbol{\mu} \in \mathbb{N}_0^{M}$, with $| \boldsymbol{\mu} |	 \leq k$, there exists a constant $C_{ \boldsymbol{\mu}} >0$ such that
	\begin{displaymath}
	\left\| D^{ \boldsymbol{\mu}} W_j \right\|_{L^{ \infty} ( \Omega_j)} \leq  \dfrac{C_{ \boldsymbol{\mu}}}{ \left(  \sup_{ \boldsymbol{x}, \boldsymbol{y} \in \Omega_j} \left\| \boldsymbol{x}-\boldsymbol{y} \right\|_2 \right)^{ | \boldsymbol{\mu}|}}, \quad j=1, \ldots, d. 
	\end{displaymath}
\end{itemize}

In what follows,  we consider the so-called \emph{Shepard's weights} \cite{Shepard68a}
\begin{equation*}
W_j\left(\boldsymbol x\right) = \dfrac{ \displaystyle \bar{W}_j \left(\boldsymbol x\right)}{  \displaystyle \sum_{k=1}^{d} \bar{W}_k \left(\boldsymbol x\right)}, \quad j=1, \ldots,d,
\end{equation*}
where $\bar{W}_j $ are compactly supported functions, with support on $\Omega_j$. 

In the next section, we describe which local approximants $L_j$ we consider here, i.e.  rational RBF expansions. However, before going into details, note that since the functions $W_j$, $j=1, \ldots, d$, form a partition of unity, if the local fits $L_j$, $j=1,\ldots,d$, satisfy the interpolation conditions then the global PU approximant trivially inherits the interpolation property. 

\section{Local rational radial basis function interpolants}
\label{rat}

For a classical RBF-PU interpolant there might be problems such as ill-conditioning (especially when the shape parameter tends to zero) and this might lead to inaccurate solutions when functions with steep gradients (or even implicit functions) are considered. These are the main reasons for which we introduce RRBF local interpolants. Indeed, as numerical evidence confirms, they are more robust for the reconstruction of 3D objects. This is also consistent with the well-known robustness of univariate rational polynomial approximation compared to the standard one. Unfortunately, differently from RRBF interpolation, the rational polynomial approximation is quite hard to extend in higher dimensions (refer e.g. to \cite{Hu}).	

On a subdomain $\Omega_j$,   we here define $L_j$ as:
\begin{equation*}
{L}_j (\boldsymbol{x}) = \dfrac{{\cal R}_j^{1}(\boldsymbol{x})}{{\cal R}_j^{2}(\boldsymbol{x})} = \dfrac{\sum_{i=1}^{N_j} \alpha^j_{i} \Phi (\boldsymbol{x}, \boldsymbol{x}_{i}^j)}{\sum_{k=1}^{N_j} \beta_{k}^j \Phi (\boldsymbol{x}, \boldsymbol{x}_{k}^j)},
\label{RRBF1}
\end{equation*}
where $N_j$ is the number of points lying on $\Omega_j$ and we assume ${\cal R}_j^{2}(\boldsymbol{x}) \neq 0$, $\boldsymbol{x} \in \Omega$.
It is easy to see that imposing the interpolation conditions leads to a  system that is underdetermined. Thus, we add extra conditions. In practice, to have a well-posed problem on $\Omega_j$, we need to look for a vector  \cite{Jakobsson}
$$
\boldsymbol{q}_j=({\cal R}_j^{2}(\boldsymbol{x}^j_1), \ldots, {\cal R}_j^{2}(\boldsymbol{x}_{N_j}^j))^T,$$
so that we can construct its relative RBF interpolant ${\cal R}_j^{2}$ in the standard way. Then, it is easy to see that, once we have $\boldsymbol{q}_j$, we are able to uniquely compute ${\cal R}_j^{1}$ such that it interpolates the function values $\boldsymbol{p}_j=D_j\boldsymbol{q}_j$, where  $D_j=\textrm{diag}(f^j_1, \ldots, f^j_N)$. To be more precise, let us assume that  $\boldsymbol{p}_j$ and $\boldsymbol{q}_j$ are given on $\Omega_j$, to compute the local RRBF interpolant, we need to solve 
\begin{equation}
\quad A_j \boldsymbol{\alpha}_j= \boldsymbol{p}_j, \quad \textrm{and} \quad A_j \boldsymbol{\beta}_j= \boldsymbol{q}_j.
\label{sisex}
\end{equation}
The existence and uniqueness of the solutions of \eqref{sisex} trivially follows from the fact that the kernel we consider is strictly positive definite. Thus, the problem now turns into the one of determining for each patch the vectors $\boldsymbol{p}_j$ and $\boldsymbol{q}_j$. 
Of course, it is reasonable to select their values such that their native space norm relative to the size of their values is as small as possible. In \cite{Jakobsson} the authors proved that this leads to define
for each patch $\Omega_j$  the vector $\boldsymbol{q}_j$ as the eigenvector  associated to the smallest eigenvalue of the problem
$\Lambda_j \boldsymbol{q}_j= \lambda_j \Theta_j \boldsymbol{q}_j,
\label{eigp1}$
with
\begin{equation*}
\Lambda_j=\dfrac{1}{||\boldsymbol{f}_j||_2^2}  D^T_j A_j^{-1} D_j +A_j^{-1}, \quad \textrm{and} \quad \Theta_j= \dfrac{1}{||\boldsymbol{f}_j||_2^2} D_j^TD_j+I_{N_j},
\label{eigp2}
\end{equation*}
where  $I_{N_j}$ is the $N_j \times N_j $ identity matrix and $A_j$ is the standard local kernel matrix:
\begin{align}
A_j = \left(
\begin{array}{cccc}
\Phi  (\boldsymbol{x}^j_1,\boldsymbol{x}^j_1 )   & \cdots  & \Phi(\boldsymbol{x}^j_1,\boldsymbol{x}_{N_j}^j  )        \\
\vdots & \ddots & \vdots   \\
\Phi  (\boldsymbol{x}^j_{N_j},\boldsymbol{x}^j_1)    & \cdots  & \Phi  (\boldsymbol{x}^j_{N_j},\boldsymbol{x}^j_{N_j} )        
\end{array}
\right).
\label{am}
\end{align}
Note that the so-constructed method is not able to handle the case where $f$ is zero at some data sites. However, for our application this is not restrictive. Indeed, we only need to carefully define the extra function values  described in Section \ref{sec1}. For instance, in what follows we take $a=2$, $b=3$ and $c=1$. Finally, note that we assume ${\cal R}^{2}(\boldsymbol{x}) \neq 0$, $\boldsymbol{x} \in \Omega$. If it does not hold, it is sufficient to impose the following constraints for the minimization problem in \eqref{eq2}: $q_i^j>0$, $i=1,\ldots,M$. Once we have such function values, we can use the method proposed in \cite{Derossi17} to obtain a positive approximant. 

Acting as explained above implies that we construct a RRBF interpolant by means of the standard RBF interpolation matrix $A_j$. This enables us to give error bounds. First note that for $\boldsymbol{x} \in \Omega$, we have that
\label{eq2}
\begin{eqnarray}
\left|f(\boldsymbol{x})- {\cal I}(\boldsymbol{x}) \right| & \leq & \sum_{j=1}^{d} \left|f(\boldsymbol{x}) - L_j(\boldsymbol{x}) \right| W_j (\boldsymbol{x}) , \nonumber \\
& \leq & \max_{j=1, \ldots, d} \left\|f-L_j \right\|_{L_{\infty}(\Omega_j)}:=\left\|f_{| \Omega_t}-L_t \right\|_{L_{\infty}(\Omega_t)}. \\
\nonumber
\end{eqnarray} 
In other words the PU approximation error is governed by the worst local error.
To formulate error bounds, we have to think of $\boldsymbol{p}_t$ and $\boldsymbol{q}_t$ as values sampled form some functions $p_t$ and $q_t$ $\in {\cal N}_{\Phi}(\Omega_t)$.
Furthermore, we also need to define the space $C_{ \nu}^{k}   ( \mathbb{R}^{M}  ) $ of all functions $f \in C^k$ whose derivatives of order $ |\boldsymbol{\mu} |=k $ satisfy $ D^{ \boldsymbol{\mu}} f \left( \boldsymbol{x} \right) = {\cal O} \left( || \boldsymbol{x} ||_2^{ \nu} \right) $ for $ || \boldsymbol{x} ||_2 \longrightarrow 0$. Then, the $\left\|f_{| \Omega_t}-L_t \right\|_{L_{\infty}(\Omega_t)}$ can be bounded by means of the following proposition. 

\begin{proposition}
	Suppose $\phi \in C_k^{\nu}(\mathbb{R}^M)$ is strictly positive definite, $p_t$ and $q_t \in {\cal N}_{\Phi}(\Omega_t)$  and let $   {\cal X}_t= \{ \boldsymbol{x}_i  ,i=1, \ldots, N_t \} \subseteq \Omega_t$, then there exists a constant $C$ independent of $h_{{\cal X}_{N_t}}$ such that
	\begin{equation*}
	|| f_{| \Omega_t} -  {L}_t||_{L_{\infty}(\Omega_t)} \leq \dfrac{ {C} h_{{\cal X}_{N_t}}^{(k+\nu)/2}}{||{\cal R}_t^{2}||_{L_{\infty}(\Omega_t)}} \left( ||f_{| \Omega_t}||_{L_{\infty}(\Omega_t)}  ||q_t||_{{\cal N}_{\Phi}(\Omega_t)} +  ||p_t||_{{\cal N}_{\Phi}(\Omega_t)}  \right) .
	\end{equation*}
	\label{thpu1}
\end{proposition}
\begin{proof}
	At first note that
	\[\arraycolsep=1.4pt\def\arraystretch{3}	
	\begin{array}{rcl}
	|| f_{| \Omega_t} -  {L}_t||_{L_{\infty}(\Omega_t)} 
	&=&\left|\left| \dfrac{{\cal R}_t^{2} f_{| \Omega_t} -{\cal R}_t^{1}}{{\cal R}_t^{2}} \right|\right|_{L_{\infty}(\Omega_t)},\\
	&=&\left| \left| \dfrac{\left({\cal R}_t^{2} f_{| \Omega_t}- q_tf_{| \Omega_t} \right) + \left(q_tf_{| \Omega_t}-{\cal R}_t^{1}  \right)}{{\cal R}_t^{2}} \right| \right|_{L_{\infty}(\Omega_t)}.\\
	\end{array} 
	\]
	Furthermore, we know that under the assumptions of this proposition, if $\Omega_t$ satisfies  an interior cone condition with constants $\tilde{\theta}$ and $\tilde{\gamma}$ and if $   {\cal X}_{N_t}= \{ \boldsymbol{x}_i  ,i=1, \ldots, N_t \} \subseteq \Omega_t$  satisfies $h_{{\cal X}_{N_t}} \leq h_0$,  there exists a constant ${C}_0$ independent of $h_{{\cal X}_{N_t}}$ and depending on $M, \tilde{\theta}$ and $ \phi$, such that (cf. \cite[Th. 11.11, p. 181]{Wendland05}) 
	\begin{equation}
	\left| \left| {\cal R}_t^{2} - q_t \right|  \right|_{L_{\infty}(\Omega_t)} \leq   {C_0} h_{{\cal X}_{N_t}}^{(k+\nu)/2} ||q_t||_{{\cal N}_{\Phi}(\Omega_t)},
	\label{eq3}
	\end{equation}
	where $h_0=\tilde{\gamma} /C_2$, with $C_2$ is from Theorem \ref{tha} applied to a local setting.
	Note that, because of the regular covering, \eqref{eq3} holds. Indeed, all the subdomains satisfy an interior cone condition and the local fill distances are uniformly bounded by the global one.
	Moreover, taking into account how $\boldsymbol{p}_t$ and $\boldsymbol{q}_t$ are related, we obtain
	\[\arraycolsep=1.4pt\def\arraystretch{2}
	\begin{array}{rcl}
	||f_{| \Omega_t} -L_t||_{L_{\infty}(\Omega_t)}
	& \leq &\dfrac{ {C} h_{{\cal X}_{N_t}}^{(k+\nu)/2} }{||{\cal R}^{2}_t||_{L_{\infty}(\Omega_t)}} \left( ||f_{| \Omega_t} ||_{L_{\infty}(\Omega_t)} ||q_t||_{{\cal N}_{\Phi}(\Omega_t)} +  ||p_t||_{{\cal N}_{\Phi}(\Omega_t)} \right) .\\
	\end{array}
	\]
\end{proof}

\begin{remark}
	Such proposition confirms that, as for the classical PU interpolant which makes use of local RBFs, the rational PU interpolant preserves the local approximation error. The bound reported in Proposition \ref{thpu1} shows  strong similarities with the ones for the standard interpolants. This could be expected, indeed, we are considering an interpolant which is essentially a rescaled classical RBF approximant. Nevertheless, despite such bounds are similar, since we rescale with a quantity depending on the largest eigenvalue of the kernel matrix, numerically we expect a more accurate computation. This effect should be more evident for kernels having a fast decay, such as the Gaussian. 
\end{remark}

\begin{remark}
	In the PU framework, an important computational issue consists in organizing points among the subdomains. To achieve this aim we use the so-called integer-based data structure, refer e.g. to \cite{Cavoretto14c,cavoretto_loocv} for further details. 
\end{remark}

\begin{remark}
	Note that the proposed scheme is also able to handle the use of anisotropic kernels. Indeed, any isotropic radial kernel can be turned into an anisotropic one by using a weighted $2$-norm instead of an unweighted one \cite{Fasshauer15}. It is enough to replace the scalar value of the shape parameter $\varepsilon$ with a symmetric positive definite matrix ${\Xi}$. More precisely, taking ${\Xi} = \textrm{diag} (\varepsilon_1, \ldots, \varepsilon_M)$  allows to chose a different scaling along the dimensions of the problem. However, since we consider quasi-uniform points, we omit further considerations or tests with anisotropic kernels.
\end{remark}

\section{Numerical experiments}
\label{ne}

The numerical experiments that follow have been carried out with \textsc{Matlab} on an Intel(R) Core(TM) i7 CPU 4712MQ 2.13 GHz processor. 

In this section we consider scattered data on a cube $\Omega=[0,\gamma]^3$, $\gamma \in \mathbb{R}^{+}$. To better assess the robustness of the method, in Subsection \ref{kf} we take known implicit functions, while in Subsection \ref{rd} we deal with unknown 3D objects.

The RBFs considered in the examples are the Gaussian $C^{\infty}$ and the Wendland's $C^2$ functions, whose formulae respectively are
\begin{equation*}
\phi_1(r)= e^{-\varepsilon^2 r^2},
\end{equation*}
and
\begin{equation*}
\phi_2(r)= (1-\varepsilon r)_{+}^4(4 \varepsilon r+1),
\end{equation*}
where $r$ is the Euclidean distance, $\varepsilon$ is the shape parameter and $(\cdot)_{+}$ denotes the truncated power function. Note that the Wendland's $C^2$ function is also used for the computation of the PU weights. We remark that the Gaussian kernel usually leads to matrices with high condition numbers, while the Wendland's $C^2$ is more stable. Therefore, the latter is strictly advised for applications with real data.

The interpolants are evaluated on a grid of  $s=80 \times 80 \times 80$  points ${\cal X}_s=\{ \bar{\boldsymbol{x}}_i, i=1,\ldots,s \}$. Moreover, to point out the accuracy, for the known functions we compute the Root Mean Square Error (RMSE), while for real data it is estimated via cross-validation.

The RRBF-PU is applied  with spherical patches whose centres are a grid of points on $\Omega$ of radius 
\begin{equation*}
\delta={\dfrac{\gamma}{d^{1/M}}},
\end{equation*}
where the number of patches $d$ is given by
\begin{equation*}
d= \left( \left\lceil \dfrac{{N}^{1/M}}{2}  \right\rceil \right)^M.
\end{equation*}

In what follows we will compare the RRBF-PU with the classical scheme based on RBFs as local approximants, i.e. the RBF-PU. The former turns out to be more robust. Furthermore, it turns out to be also efficient (refer to \cite{rat0}). This is due to the fact that for the RRBF-PU the eigenvectors of the local kernel matrices are calculated by means of the DACG scheme
\cite{BergamaschiPutti02}. Indeed,  DACG  has been shown to be faster than the Lanczos method (standard \textsc{Matlab} routine {\tt eigs})  \cite{Arnoldi} when a small number of eigenpairs are being sought. 

Finally, comparisons with the BLOOCV-PU will be also carried out. Such technique, studied in \cite{cavoretto_loocv}, is based on an optimal selection of both the radius and shape parameter for each patch. We will show that, provided that points are not clustered, the RRBF-PU is competitive, especially because of its efficiency.

\subsection{Experiments with artificial data}
\label{kf}

The first test function we consider in these examples defines the easiest 3D object, i.e. a sphere. In particular,
\begin{equation*}
f_1(x_1,x_2,x_3)=(x_1-0.5)^2+(x_2-0.5)^2+(x_3-0.5)^2-0.5^2=0.
\end{equation*}
The second test is instead carried out by considering 
\begin{equation*}
f_2(x_1,x_2,x_3)=f_1(x_1,x_2,x_3)+\sin^4(4y)=0.
\end{equation*}

As data, we take four sets of random nodes ($n=1089,4225,16641, 66049$). An example of $1089$ data describing the surfaces is plotted in Figure \ref{f1f2}. 

\begin{figure}[ht!] 
	\begin{center}
		\makebox[\textwidth]{
			\includegraphics[height=.27\textheight]{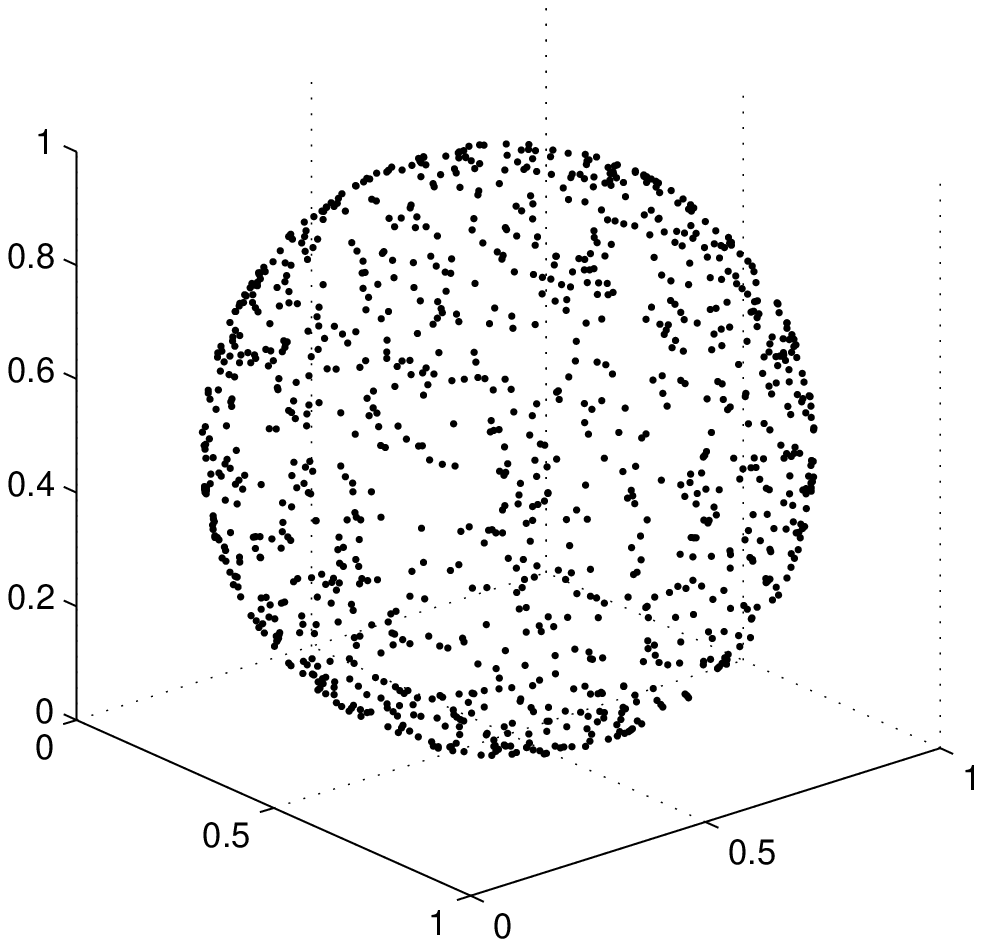} 
			\hskip -1.4cm
			\includegraphics[height=.27\textheight]{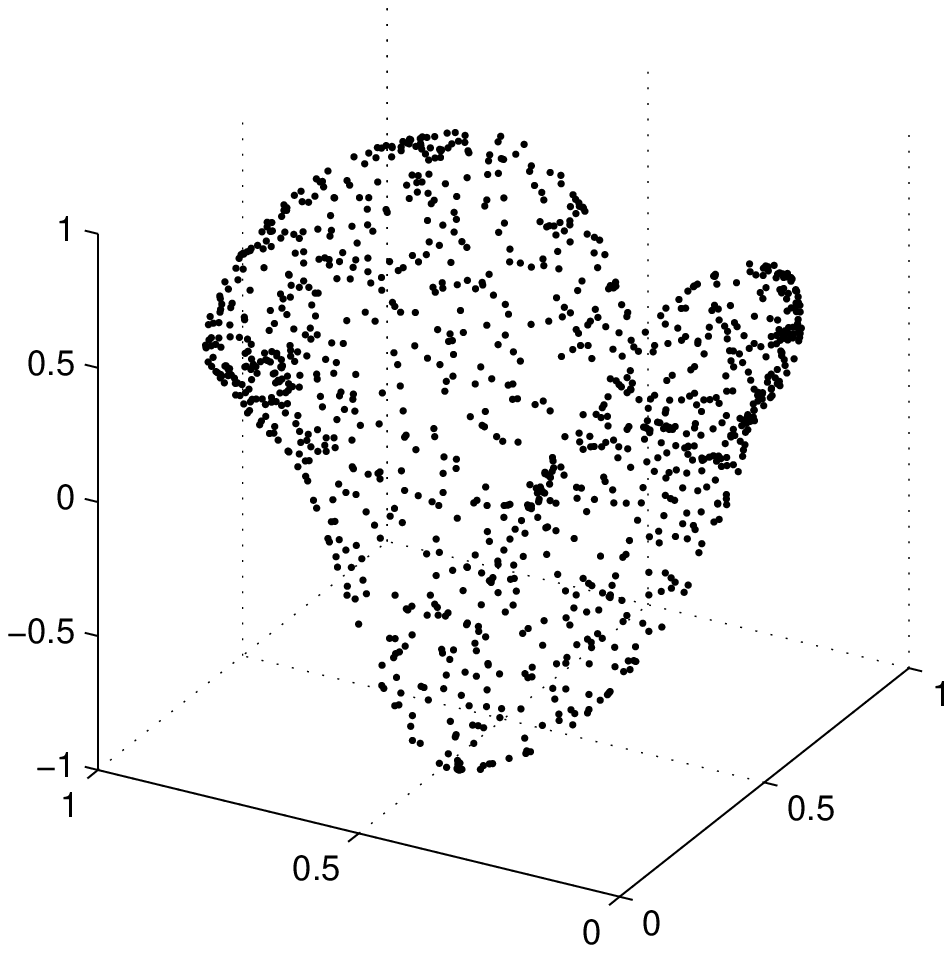}}
		\caption{Examples of point cloud data sets for $f_1$ (left) and $f_2$ (right).}
		\label{f1f2}
	\end{center}
\end{figure}	

We remark that, for a given set of $n$ data, we first compute and consistently orient the surface  normals so that we construct the augmented sets ${\cal X}_N$ and ${\cal F}_N$, as shown in Section \ref{sec1}, where $N \approx 3 n$. Then, we compute the RRBF-PU for these augmented sets and  we  evaluate the iso-surface corresponding to the original set of nodes.

In this subsection, we consider the Gaussian kernel as basis function. To  assess the behaviour of the error with respect to the shape parameter we evaluate  for  $f_1$ and $f_2$ the RMSEs (obtained with both the RBF-PU and RRBF-PU) for $20$ values of the shape parameter $\varepsilon$  in the range $[10^{-3}, 10^{2}]$.  Refer to Figure \ref{f7f8}.

\begin{figure}[ht!] 
	\begin{center}
		\makebox[\textwidth]{
			\includegraphics[height=.27\textheight]{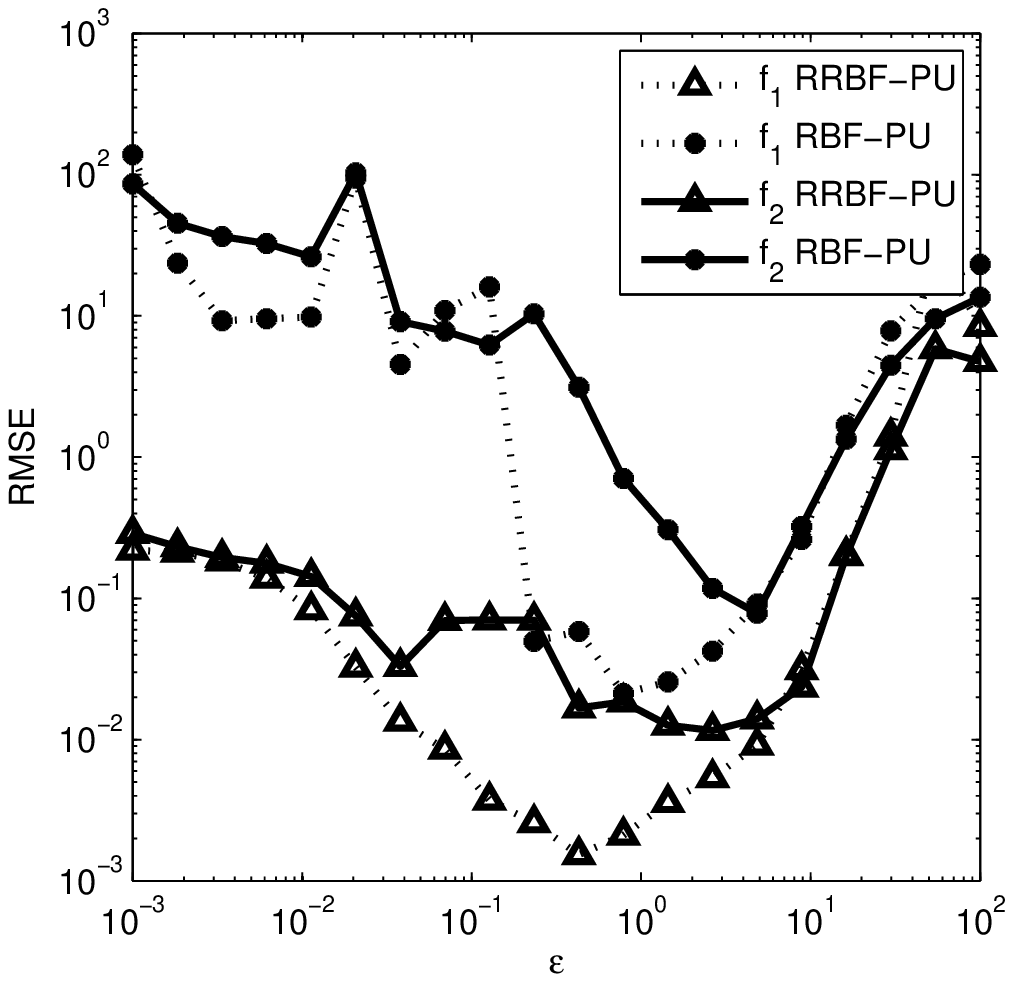} 
			\hskip -1.2cm
			\includegraphics[height=.27\textheight]{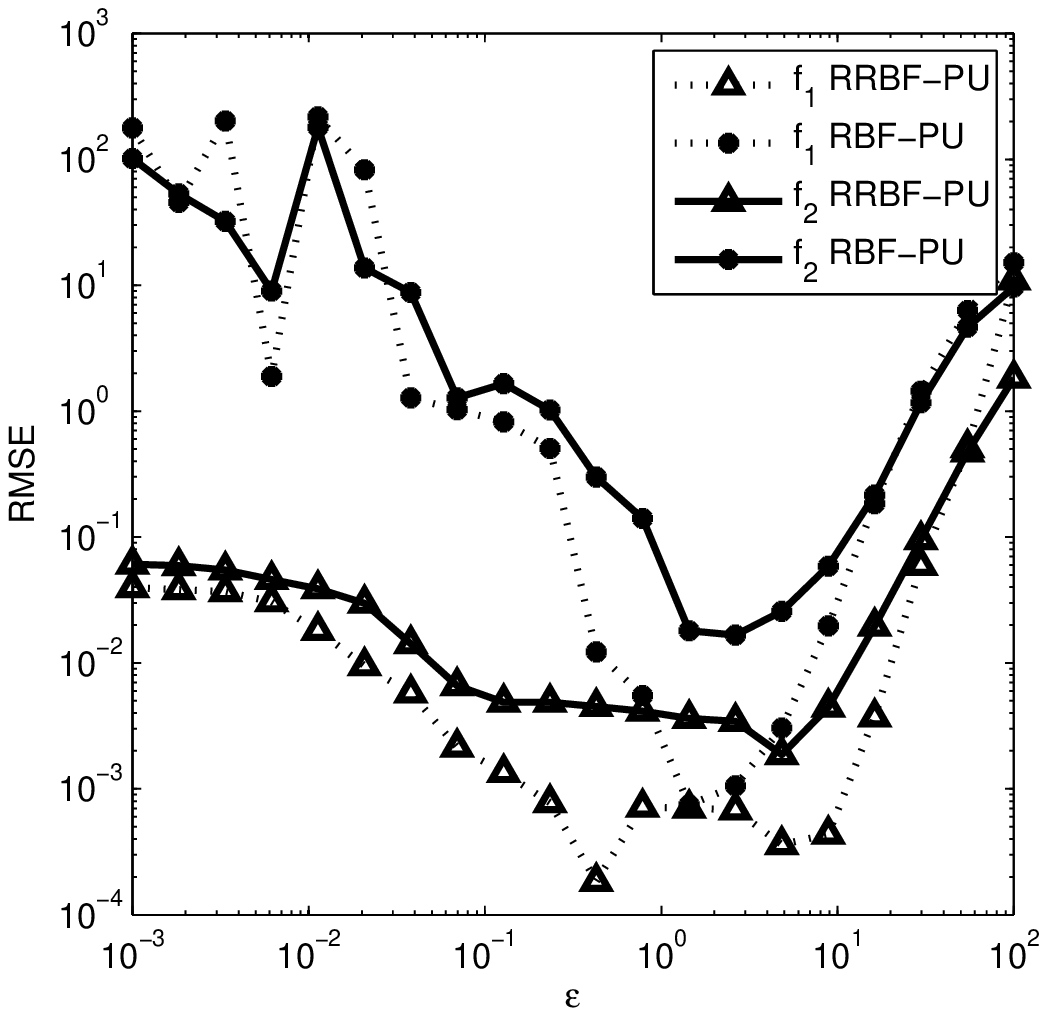}}
		\makebox[\textwidth]{
			\includegraphics[height=.27\textheight]{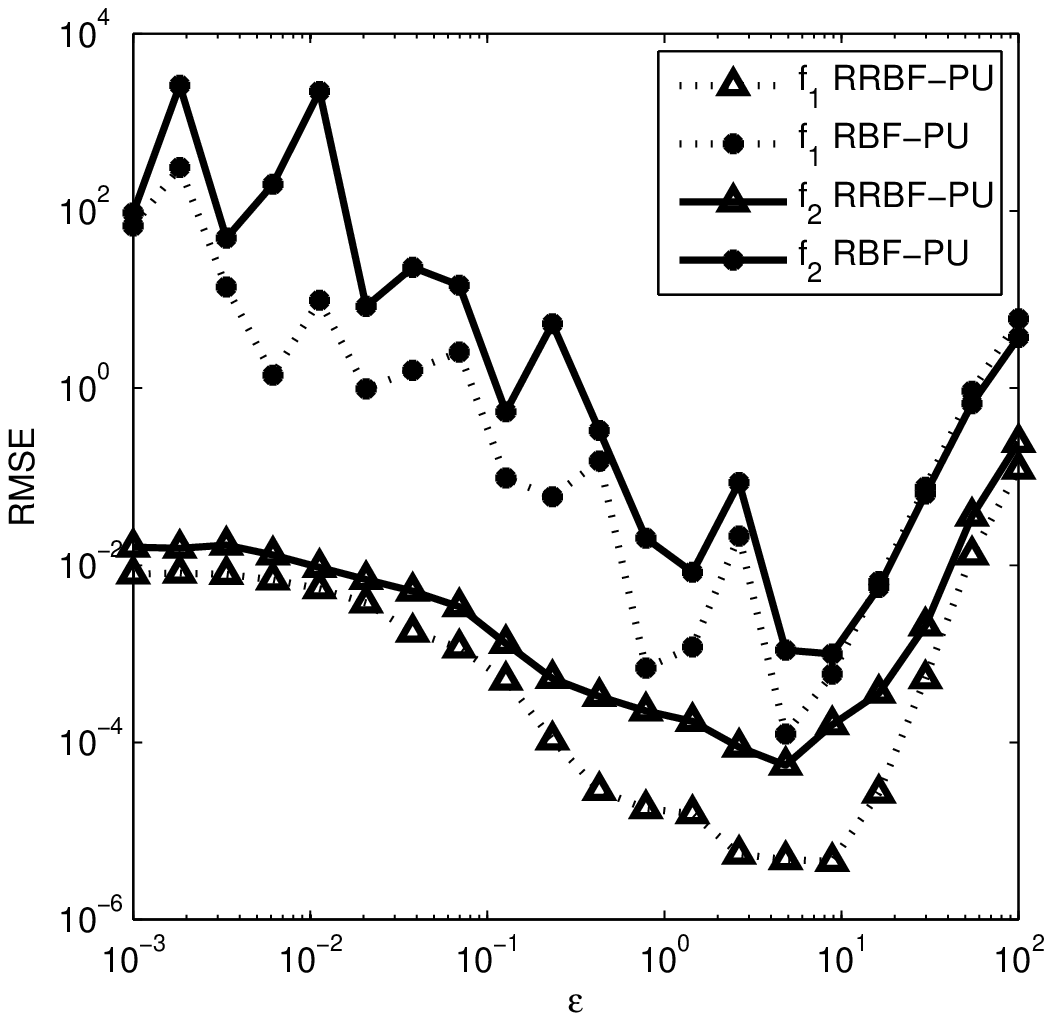} 
			\hskip -1.2cm
			\includegraphics[height=.27\textheight]{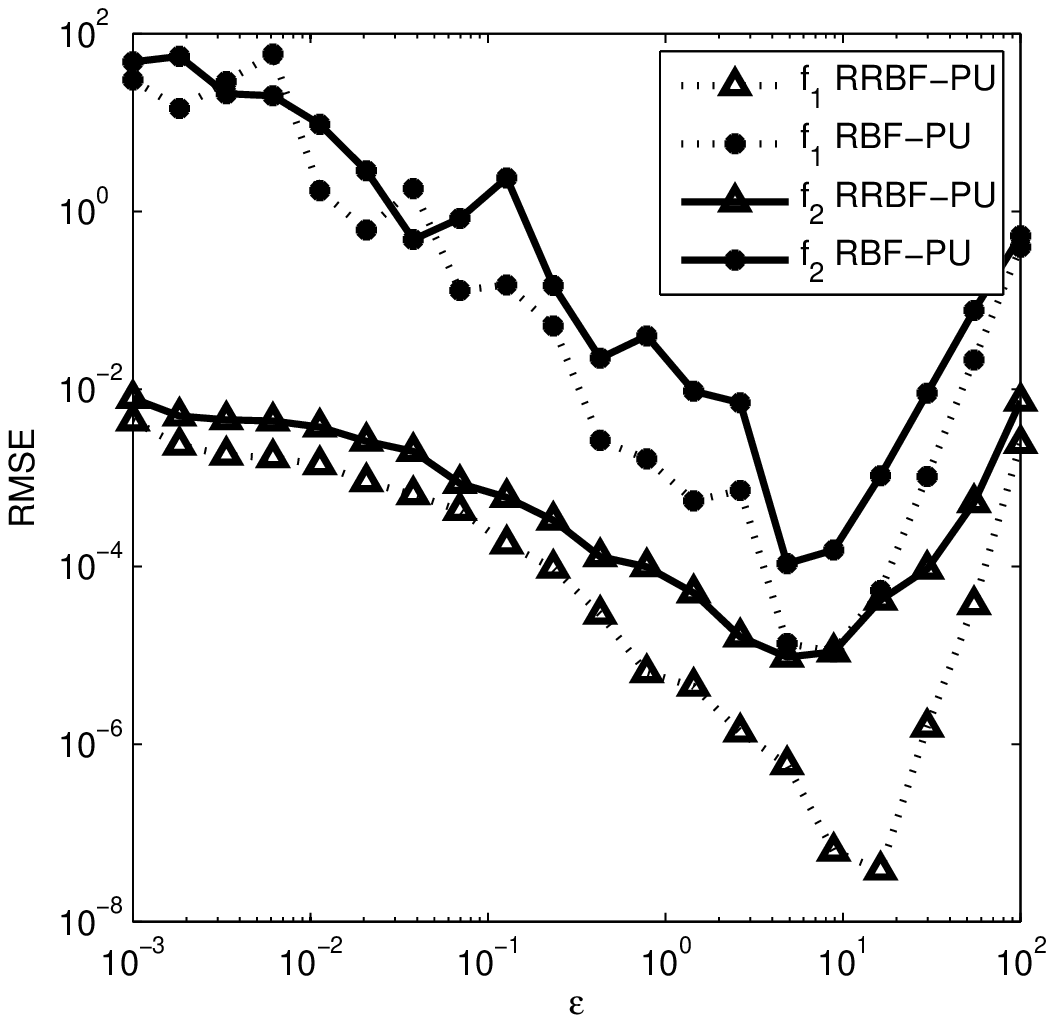}}			
		\caption{RMSEs obtained by varying $\varepsilon$ for the Gaussian $C^\infty$ kernel.  From left to right, top to bottom, we consider $n=$ $1089$, $4225$,  $16641$ and $66049$ random data. The dotted line represents the RMSEs (marked with dot and  triangle for RBF-PU and RRBF-PU, respectively) for $f_1$. The continuous line represents the RMSEs (marked with dot and  triangle for RBF-PU and RRBF-PU, respectively) for $f_2$.}
		\label{f7f8}
	\end{center}
\end{figure}	

Such study reveals that the RRBF-PU is more effective than the standard one. Furthermore, in Tables \ref{tab_1}--\ref{tab_2}, we also report the RMSE in correspondence of the optimal shape parameter for $f_1$ and $f_2$, respectively. It is evident that, especially for large values of $n$ the RRBF-PU outperforms the standard one. 
Finally, note that in Table \ref{tab_1}, we also report the results obtained via BLOOCV-PU. As expected, since it is able to find the optimal radius and the optimal shape parameter for each patch,  it turns out to be more accurate. Nevertheless, while the difference among RRBF-PU and BLOOCV-PU is not so evident in terms of accuracy, it is truly marked from the point of view of the efficiency. For instance, the classical RBF-PU  takes $3.96$ s to approximate a data set consisting of $n=1089$ points. The RRBF-PU, thanks to the DACG algorithm,  only requires $4.59$ s, while the BLOOCV-PU needs $35.2$ s. Thus, for quasi-uniform nodes, the RRBF-PU turns out to be accurate and efficient, while (because of the optimization process) the BLOOCV-PU has a high computational cost. Nevertheless, we remark that when points are clustered, then the use of the BLOOCV-PU is meaningful. Indeed, in this case, a suitable selection of the patch radius is essential to avoid loss of convergence orders on several subdomains.

\begin{table}[ht!]
	\begin{center}
		\begin{tabular}{cccc} \hline
			\rule[0mm]{0mm}{3ex}
			$n$  & $\varepsilon^*$ & Method & RMSE  \\
			\hline 
			\rule[0mm]{0mm}{3ex}
			$1089$   & $0.42$ & RBF-PU  & $1.67{\rm E}-2$   \\
			& -- & BLOOCV-PU  & $8.49{\rm E}-4$   \\
			& $0.42$ & RRBF-PU & $1.53{\rm E}-3$   \\	
			\rule[0mm]{0mm}{3ex}
			$4225$   & $1.43$ &  RBF-PU  & $7.45{\rm E}-4$   \\
			& -- & BLOOCV-PU  & $7.53{\rm E}-5$   \\
			& $0.42$ & RRBF-PU & $1.84{\rm E}-4$   \\	
			\rule[0mm]{0mm}{3ex}
			$16641$   & $4.83$ &  RBF-PU  & $1.23{\rm E}-4$   \\
			& -- & BLOOCV-PU  & $2.43{\rm E}-6$   \\
			& $8.85$ & RRBF-PU & $4.52{\rm E}-6$   \\
			\rule[0mm]{0mm}{3ex}
			$66049$  &  $8.85$ &  RBF-PU  & $1.10{\rm E}-5$   \\
			& -- & BLOOCV-PU  & $2.17{\rm E}-8$   \\
			& $16.24$ & RRBF-PU  & $3.80{\rm E}-8$   \\
			\hline 
		\end{tabular}
	\end{center}
	\caption{RMSEs for the optimal shape parameter $\varepsilon^*$ obtained for the test function $f_1$ and several sets of random nodes.  }
	\label{tab_1}
\end{table}

\begin{table}[ht!]
	\begin{center}
		\begin{tabular}{cccc} \hline
			\rule[0mm]{0mm}{3ex}
			$n$  & $\varepsilon^*$ & Method & RMSE  \\
			\hline 
			\rule[0mm]{0mm}{3ex}
			$1089$   & $4.83$ & RBF-PU  & $7.84{\rm E}-2$   \\
			& $2.63$ & RRBF-PU & $1.16{\rm E}-2$   \\	
			\rule[0mm]{0mm}{3ex}
			$4225$   & $2.63$ &  RBF-PU  & $1.66{\rm E}-2$   \\
			& $4.83$ & RRBF-PU & $1.86{\rm E}-3$   \\	
			\rule[0mm]{0mm}{3ex}
			$16641$   & $8.85$ &  RBF-PU  & $9.97{\rm E}-4$   \\
			& $4.83$ & RRBF-PU & $5.48{\rm E}-5$   \\
			\rule[0mm]{0mm}{3ex}
			$66049$  &  $4.83$ &  RBF-PU  & $1.08{\rm E}-4$   \\
			& $4.83$ & RRBF-PU  & $9.52{\rm E}-6$   \\
			\hline 
		\end{tabular}
	\end{center}
	\caption{RMSEs for the optimal shape parameter $\varepsilon^*$ obtained for the test function $f_2$ and several sets of random nodes.}
	\label{tab_2}
\end{table}

\subsection{Experiments with real data}
\label{rd}

Here we show with some numerical experiments the flexibility of the RRBF-PU for the reconstruction of 3D objects.
The data sets used in the following examples (available at  {\tt http://graphics.stanford.edu/data/3Dscanrep/}) correspond to various point cloud data set of the well-known \emph{Stanford Bunny} for $n=$ $453$, $1089$, $8171$ and $35974$.

The RBF used to approximate the 3D object is the Wendland's $C^2$  function.
In Figure \ref{f13f14}, we show the  graphical results for both RBF-PU and RRBF-PU of using $n=35974$ points. Also in this case, we recover the pattern already discovered about the fact that the RRBF-PU turns out to be more effective than the classical one.

\begin{figure}[ht!] 
	\begin{center}
		\makebox[\textwidth]{
			\includegraphics[height=.27\textheight]{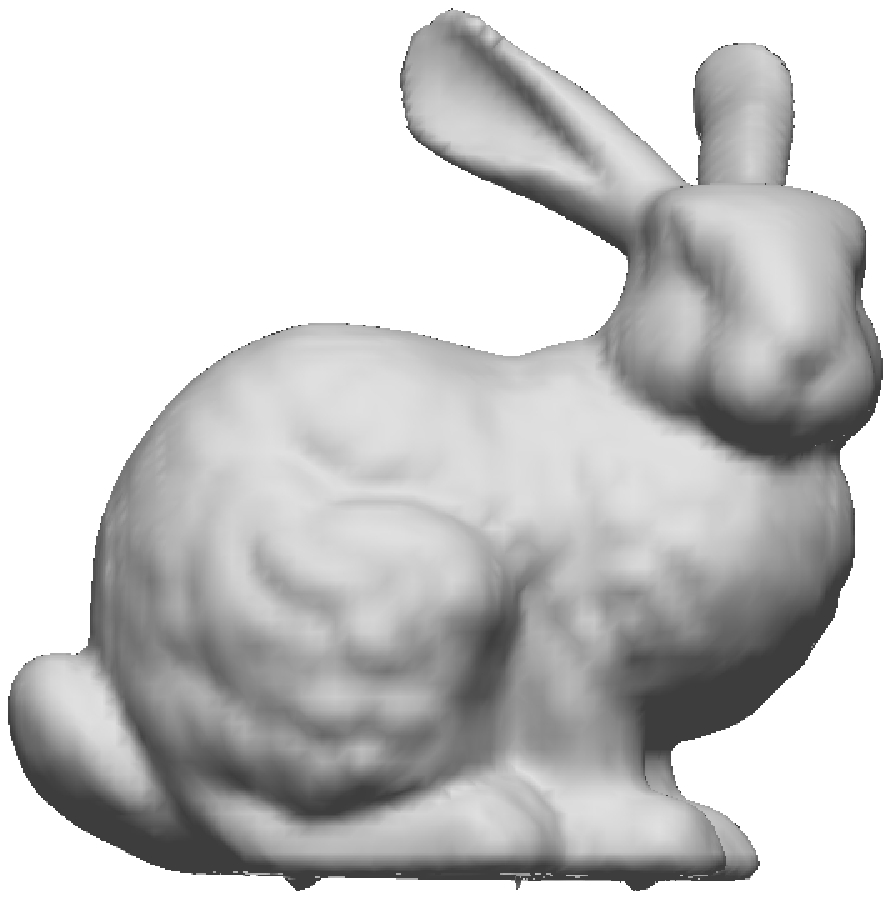} 
			\hskip -1.4cm
			\includegraphics[height=.27\textheight]{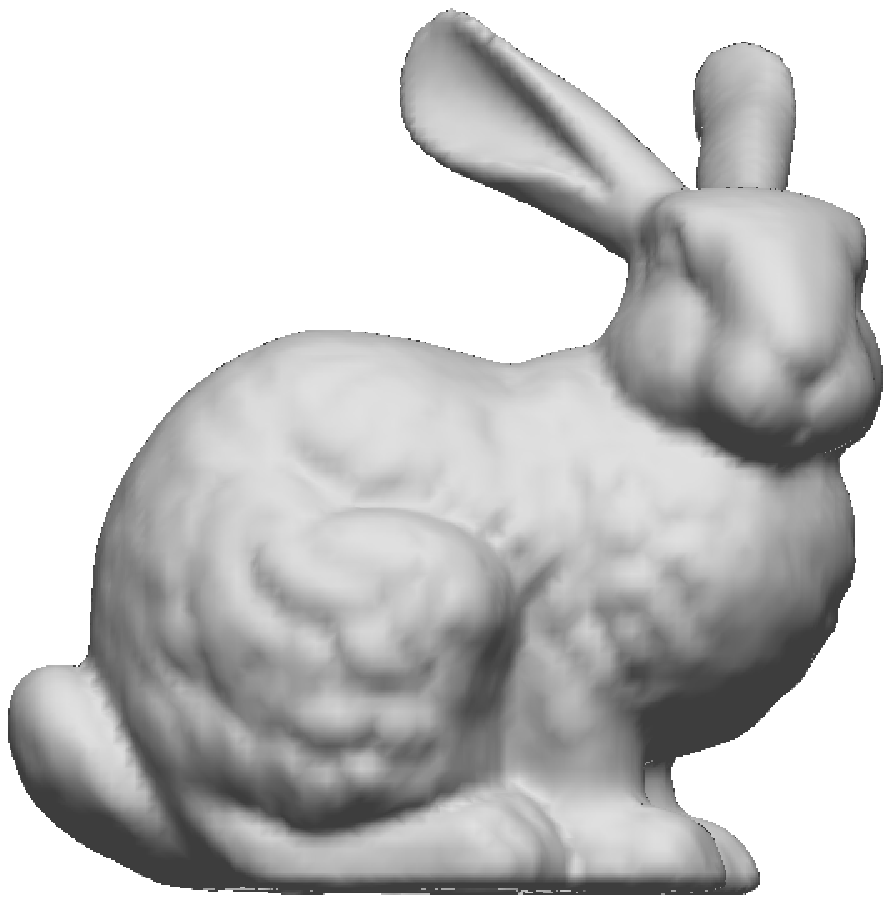}}
		\caption{The Stanford Bunny with $35947$ points reconstructed via the RBF-PU (left) and RRBF-PU (right) with $\varepsilon=1$.}	
		\label{f13f14}
	\end{center}
\end{figure}

To conclude,  as a confirm of the graphical results previously shown, we also report the estimated errors (via cross-validation) for the different data sets in Table \ref{tab_3}. In this case we fix $\varepsilon=1$. 

\begin{table}[ht!]
	\begin{center}
		\begin{tabular}{ccc} \hline
			\rule[0mm]{0mm}{3ex}
			$n$   & Method & RMSE  \\
			\hline 
			\rule[0mm]{0mm}{3ex}
			$453$   & RBF-PU  & $6.61{\rm E}-2$   \\
			& RRBF-PU & $4.41{\rm E}-2$   \\	
			\rule[0mm]{0mm}{3ex}
			$1889$   &  RBF-PU  & $5.17{\rm E}-2$   \\
			& RRBF-PU & $1.12{\rm E}-2$   \\	
			\rule[0mm]{0mm}{3ex}
			$8171$   &  RBF-PU  & $1.21{\rm E}-1$   \\
			& RRBF-PU & $6.35{\rm E}-3$   \\
			\rule[0mm]{0mm}{3ex}
			$35947$   &  RBF-PU  & $9.65{\rm E}-3$   \\
			& RRBF-PU  & $3.53{\rm E}-3$   \\
			\hline 
		\end{tabular}
	\end{center}
	\caption{The estimated (via cross validation) RMSEs for varius data sets of the Stanford Bunny.}
	\label{tab_3}
\end{table}		

\section{Final remarks}
\label{concl}						

This investigation reveals that the  RRBF-PU can be used as effective and efficient tool for the approximation of 3D objects. It takes advantage of being meshfree and more robust than a standard approach. 

Thus, as future work we need to carry out studies for coupling this scheme with the well-known stable methods (see e.g. \cite{Demarchi15,Fasshauer15,Larsson-Lehto}) and further investigations about the Lebesgue constant are also essential. To achieve this aim, we need to study the cardinal form of the rational expansion.

\section{Acknowledgments}     	
This research has been accomplished within Rete ITaliana di Approssimazione (RITA) and supported by:
\begin{itemize}
	\item GNCS-INdAM,
	\item the research project \emph{Radial basis functions approximations: stability issues and applications}, 
	No. BIRD167404. 
\end{itemize}

\end{document}